\documentclass[12pt,a4paper]{amsart}
\usepackage{mathrsfs}
\usepackage{amssymb,amsmath,amsthm,color}
\usepackage{graphicx,mcite}
\usepackage{hyperref}
\usepackage{url}
\usepackage{setspace}

\newcommand{\loc}{\mbox{loc}}

\newtheorem{theorem}{Theorem}[section]
\newtheorem{lemma}[theorem]{Lemma}
\newtheorem{proposition}[theorem]{Proposition}
\newtheorem{corollary}[theorem]{Corollary}

\theoremstyle{definition}

\newtheorem{remark}[theorem]{Remark}

\numberwithin{equation}{section}

\begin{document}

\title[Injectivity of twisted spherical means]
{Coxeter system of lines are sets of injectivity for the twisted spherical means on $\mathbb C$}

\author{ Rajesh K. Srivastava}
\address{School of Mathematics, Harish-Chandra Research Institute, Allahabad, India 211019.}
\email{rksri@hri.res.in}


\date{\today}

\dedicatory{Dedicated to Prof. E. M. Stein on the occasion of his 80th birthday.}


\begin{abstract}
It is well known that a line in $\mathbb R^2$ is not a set of injectivity for
the spherical means for odd functions about that line. We prove that any line passing
through the origin is a set of injectivity for the twisted spherical means (TSM) for functions
$f\in L^2(\mathbb C),$ whose each of spectral projection $ e^{\frac{1}{4}|z|^2}f\times\varphi_k$
is a polynomial. Then, we prove that any Coxeter system of even number of lines is a set of
injectivity for the TSM for $L^q(\mathbb C),~1\leq q\leq2.$

\end{abstract}

\maketitle

\section{Introduction}\label{section1}
In an interesting result, Courant and Hilbert (\cite{CH}, p. 699) had proved that if
the circular averages of a function $f$ which is even with respect to a line $L,$ vanishes
over all circles centered at points of $L,$ then $f\equiv0.$
As a consequence of this result, the circular averages of a function $f$ vanish over all circles
centered at points of $L$ if and only if $f$ is odd with respect to $L,$ (see \cite{AQ}, Lemma 6.3).
Hence, any line $L$ in $\mathbb R^2$ is not a set of injectivity for the spherical means for
the odd functions about $L.$

\vspace{.021in}

However, these result does not continue to hold for injectivity of the
twisted spherical means on complex plane $\mathbb C,$ because of non-commutative
nature of underlying geometry of the Heisenberg group, (see \cite{BP, CCG, CDPT, H}).
The question, in general that any real analytic curve can be a set of injectivity
for the twisted spherical mean for $L^2(\mathbb C),$ is still an open problem.
However, we are able to prove the following partial results for the TSM.

\vspace{.021in}

Let $f\in L^2(\mathbb C)$ be such that for each $k\in\mathbb Z_{+}$ (set of non-negative integers),
the projection $e^{\frac{1}{4}|z|^2}f\times\varphi_k$ is a polynomial. This space consists of these
functions $f$ is much larger than $U(1)$- finite functions in $L^2(\mathbb C).$ Suppose the twisted spherical
mean $f\times\mu_r(x)$ of the function $f$ vanishes $\forall r>0$ and $\forall x\in\mathbb R.$
Then $f=0$ a.e. That is, the $X$-axis is a set of injectivity for the TSM on $\mathbb C.$

By rotation, it follows that any line passing through the origin is a set of injectivity for
the TSM. Since $f\times\varphi_k$ is a real analytic function, in the above case, we only need
the centers to be a sequence in $\mathbb R$ having a limit point.

\vspace{.021in}

With the same exponential condition, we observe that any curve $\gamma(t)=(\gamma_1(t),\gamma_2(t)),$
which passes through the origin, where $\gamma_j,~j=1,2$  are polynomials is also a set of
injectivity for the TSM. It is an interesting question, wether a real analytic curve is a set of
injectivity for the TSM for a smooth class of functions.

\vspace{.021in}

Further, to complete the arguments of our idea, we prove that any two perpendicular
lines is a set of injectivity for the TSM on $L^q(\mathbb C).$ Moreover, this result
implies that any Coxeter system of even number of lines is a set of injectivity for
the TSM on $L^q(\mathbb C).$ These results for the twisted spherical means are in
sharp contrast to the well known result for injectivity of the Euclidean spherical
means  on $\mathbb R^2,$ due to Agranovsky and Quinto \cite{AQ}.

\vspace{.021in}

On the basis of these striking results, it is therefore natural to ask, whether
any Coxeter system of odd number of lines can be a set of injectivity for the TSM.
We believe, our techniques with slight modifications would continue to work for any
Coxeter system of lines to be a set of injectivity for the TSM, which we leave open
for the time being.

\vspace{.021in}
In 1996, Agranovsky and Quinto have proved a major breakthrough result in the
\emph{integral geometry,} which completely characterizes the sets of injectivity
for the spherical means on the space of compactly supported
continuous functions on $\mathbb R^2.$ Their result says that the exceptional set
for the sets of injectivity is a very thin set which consists of a Coxeter system
of lines union finitely many points.

\vspace{.021in}

Let $\mu_r$ be the normalized surface measure on sphere $S_r(x).$ Let $\mathscr F\subseteq L^1_{\loc}(\mathbb R^n).$
We say that $S\subseteq\mathbb R^n$ is a set of injectivity for the spherical means for $\mathscr F$
if for $f\in\mathscr  F$ with $f\ast\mu_r(x)=0, \forall r>0$ and $\forall x\in S,$ implies $f=0$ a.e.

\begin{theorem}\label{th7}\cite{AQ}
A set $S\subset\mathbb R^2$ is a set of injectivity for the spherical means for
$C_c(\mathbb R^2)$ if and only if $S\nsubseteq\omega(\Sigma_N)\cup F,$ where $\omega$
is a rigid motion of $\mathbb R^2, \Sigma_N=\cup_{l=0}^{N-1}\{te^{\frac{i\pi l}{N}}: t\in\mathbb R\}$
is a  Coxeter system of $N$ lines and $F$ is a finite set in $\mathbb R^2.$
\end{theorem}
In particular, any closed curve is a set of injectivity for $C_c(\mathbb R^2).$
In fact, Agranovsky et al. \cite{ABK} further prove that the boundary of any bounded domain
in $\mathbb R^n~(n\geq2)$ is set of injectivity for the spherical means on
$L^p(\mathbb R^n),$ with $1\leq p\leq\frac{2n}{n-1}.$ For $p>\frac{2n}{n-1},$
unit sphere $S^{n-1}$ is an example of non-injectivity set in $\mathbb R^n.$
This result has been generalized for certain weighted spherical means, see \cite{NRR}.
In general, the question of set of injectivity for the spherical means with real analytic
weight is still open. In \cite{NRR}, it has been shown that $S^{n-1}$ is a set of
injectivity for the spherical means with real analytic weights for the class of radial
functions.

\smallskip

An analogue of Theorem \ref{th7} in the higher dimensions is still open and appeared as a
conjecture in their work \cite{AQ}. It says that the sets of non-injectivity for the Euclidean
spherical means are contained in a certain algebraic variety. Following is their conjecture.

\smallskip

\noindent{\bf Conjecture \cite{AQ}.}
A set $S\subset\mathbb R^n$ is a set of injectivity for the spherical means for $C_c(\mathbb R^2)$
if and only if $S\nsubseteq\omega(\Sigma)\cup F,$ where $\omega$ is a rigid motion of $\mathbb R^n,$
$\Sigma$ is the zero set of a homogeneous harmonic polynomial and $F$ is an algebraic variety in
$\mathbb R^n$ of co-dimension at most $2.$

\smallskip

This conjecture remains unsolved, however a partial result related to this conjecture has
been proved by Kuchment et al. \cite{AK}. They also present a brief survey on the recent
development towards the above conjecture. However, in this article, we observe that this
conjecture does not continue to hold for the spherical means on the Heisenberg group
$\mathbb H^1=\mathbb C\times\mathbb R.$ In fact result on $\mathbb H^1$ is an adverse to
the Euclidean result, Theorem \ref{th7} on $\mathbb R^2.$

\bigskip

In more general, let $f\in L^1_{\loc}(\mathbb C^n)$ and write $S(f)=\{z\in\mathbb C^n:
f\times\mu_r(z)=0, \forall r>0\}.$ Our main problem is to describe completely the geometrical
structure of $S(f)$ that would characterize which ``sets" are set of injectivity for the TSM.
For example, let $f$ be a non zero type function $f(z)=\tilde{a}(|z|)P(z)\in L^2(\mathbb C^n)\cap C(\mathbb C^n),$
where $P\in H_{p,q}.$ Here $H_{p,q}$ is the space of homogeneous harmonic polynomials on $\mathbb C^n$
of type
\[P(z)=\sum_{|\alpha|=p,|\beta|=q}C_{\alpha\beta}z^\alpha\bar z^\beta.\]
Then $S(f)=P^{-1}(0)\cup F,$ where $F$ is the union of finitely many spheres
centered at the origin. This means a set $S\subset\mathbb C^n$ is set
of injectivity for twisted spherical means for type functions if and only if $S\nsubseteq P^{-1}(0).$
Since $P$ is harmonic, by maximal principle $P^{-1}(0)$ can not contain the boundary of any
bounded domain in $\mathbb C^n.$ Hence the boundary of any bounded domain would be a possible
candidate for set of injectivity for the TSM. The question that the boundary of the bounded
domain is a set of injectivity for the TSM has been taken up by many authors (see \cite{AR, NT1, ST}).
In a result of Narayanan and Thangavelu \cite{NT1}, it has been proved that the spheres centered
at the origin are sets of injectivity for the TSM on $\mathbb C^n.$ The author has generalized
the result of \cite{NT1} for certain weighted twisted spherical means, (see \cite{Sri}).
In general, the question of set of injectivity for the twisted spherical means (TSM)
with real analytic weight is still open.

\section{Notation and Preliminaries}\label{section2}
We define the twisted spherical means which arise in the study of spherical means on Heisenberg
group. The group $\mathbb H^n,$ as a manifold, is $\mathbb C^n \times\mathbb R$ with the group law
\[(z, t)(w, s)=(z+w,t+s+\frac{1}{2}\text{Im}(z.\bar{w})),~z,w\in\mathbb C^n\text{ and }t,s\in\mathbb R.\]
Let $\mu_s$ be the normalized surface measure on the sphere $ \{ (z,0):~|z|=s\} \subset \mathbb H^n.$
The spherical means of a function $f$ in $L^1(\mathbb H^n)$ are defined by
\begin{equation} \label{exp22}
f\ast\mu_s(z, t)=\int_{|w|=s}~f((z,t)(-w,0))~d\mu_s(w).
\end{equation}
Thus the spherical means can be thought of as convolution operators. An important technique in many
problem on $\mathbb H^n$ is to take partial Fourier transform in the $t$-variable to reduce matters
to $\mathbb C^n$. Let
\[f^\lambda(z)=\int_\mathbb R f(z,t)e^{i \lambda t} dt\]
be the inverse Fourier transform of $f$ in the $t$-variable. Then a simple calculation shows that

\begin{eqnarray*}
(f \ast \mu_s)^\lambda(z)&=&\int_{-\infty}^{~\infty}~f \ast \mu_s(z,t)e^{i\lambda t} dt\\
&=&\int_{|w| = s}~f^\lambda (z-w)e^{\frac{i\lambda}{2} \text{Im}(z.\bar{w})}~d\mu_s(w)\\
&=&f^\lambda\times_\lambda\mu_s(z),
\end{eqnarray*}
where $\mu_s$ is now being thought of as normalized surface measure
on the sphere $S_s(o)=\{z\in\mathbb C^n: |z|=s\}$ in $\mathbb C^n.$
Thus the spherical mean $f\ast \mu_s$ on the Heisenberg group can be
studied using the $\lambda$-twisted spherical mean $f^\lambda
\times_\lambda\mu_s$ on $\mathbb C^n.$ For $\lambda \neq 0,$
a further scaling argument shows that it is enough to study these
means for the case of $\lambda=1.$

Let $\mathscr F\subseteq L^1_{\loc}(\mathbb C^n).$  We say $S\subseteq\mathbb C^n$
is a set of injectivity for twisted spherical means for $\mathscr F$ if for
$f\in\mathscr F$ with $f\times\mu_r(z)=0, \forall r>0$ and $\forall z\in S,$
implies $f=0$ a.e.  The results on set of injectivity differ in the choice
of sets and the class of functions considered. We would like to refer to
\cite{AR, NT1, Sri}, for some results on the sets of injectivity for the TSM.

\bigskip

We need the following basic facts from the theory of bigraded
spherical harmonics (see \cite{T}, p.62 for details). We shall use
the notation $K=U(n)$ and $M=U(n-1).$ Then, $S^{2n-1}\cong K/M$ under
the map $kM\rightarrow k.e_n,$ $k\in U(n)$ and $e_n=(0,\ldots
,1)\in \mathbb C^n.$ Let $\hat{K}_M$ denote the set of all
equivalence classes of irreducible unitary representations of $K$
which have a nonzero $M$-fixed vector. It is known that each
representation in $\hat{K}_M$ has a unique nonzero $M$-fixed vector,
up to a scalar multiple.

For a $\delta\in\hat{K}_M,$ which is realized on $V_{\delta},$ let
$\{e_1,\ldots, e_{d(\delta)}\}$ be an orthonormal basis of
$V_{\delta}$ with $e_1$ as the $M$-fixed vector. Let
$t_{ij}^{\delta}(k)=\langle e_i,\delta (k)e_j \rangle ,$ $k\in K$
and $\langle , \rangle$ stand for the innerproduct on $V_{\delta}.$
By Peter-Weyl theorem, it follows that $\{\sqrt{d(\delta
)}t_{j1}^{\delta}:1\leq j\leq d(\delta ),\delta\in\hat{K}_M\}$ is an
orthonormal basis of $L^2(K/M)$ (see \cite{T}, p.14 for details).
Define $Y_j^{\delta} (\omega )=\sqrt{d(\delta )}t_{j1}^{\delta}(k),$
where $\omega =k.e_n\in S^{2n-1},$ $k \in K.$ It then follows that
$\{Y_j^{\delta}:1\leq j\leq d(\delta ),\delta\in \hat{K}_M, \}$
forms an orthonormal basis for $L^2(S^{2n-1}).$

For our purpose, we need a concrete realization of the
representations in $\hat{K}_M,$ which can be done in the following
way. See \cite{Ru}, p.253, for details.
For $p,q\in\mathbb Z_+$, let $P_{p,q}$ denote the space of all
polynomials $P$ in $z$ and $\bar{z}$ of the form
$$
P(z)=\sum_{|\alpha|=p}\sum_{|\beta|=q}c_{\alpha\beta} z^\alpha
\bar{z}^\beta.
 $$
Let $H_{p,q}=\{P\in P_{p,q}:\Delta P=0\}.$ The elements of $H_{p,q}$ are
called the bigraded solid harmonics on $\mathbb C^n.$ The group  $K$
acts on $H_{p,q}$ in a natural way. It is easy to see that the space
$H_{p,q}$ is $K$-invariant. Let $\pi_{p,q}$ denote the corresponding
representation of $K$ on $H_{p,q}.$ Then representations in
$\hat{K}_M$ can be identified, up to unitary equivalence, with the
collection $\{\pi_{p,q}: p,q \in \mathbb Z_+\}.$

Define the bigraded spherical harmonic by $Y_j^{p,q}(\omega)=\sqrt{d(p,q )}t_{j1}^{p,q}(k).$
Then $\{Y_j^{p,q}:1\leq j\leq d(p,q),p,q \in \mathbb Z_+ \}$ forms an orthonormal basis for
$L^2(S^{2n-1}).$ Therefore, for a continuous function $f$ on $\mathbb C^n,$ writing
$z=\rho \,\omega,$ where $\rho>0$ and $\omega \in S^{2n-1},$ we can expand the function $f$
in terms of spherical harmonics as
\begin{equation}\label{Bexp4}
f(\rho\omega)=\sum_{p,q\geq0}\sum_{j=1}^{d(p,q)}a_j^{p,q}(\rho)Y_j^{p,q}(\omega),
\end{equation}
where the series on the right-hand side converges uniformly on every compact set
$K\subseteq\mathbb C^n.$ The functions $a_j^{p,q}$ are called the spherical
harmonic coefficients of $f$ and function $a^{p,q}(\rho)Y^{p,q}(\omega)$ is
known as the type function.

\smallskip

We need the Hecke-Bochner identity for the spectral projections $f\times\varphi_k^{n-1},$
for function  $f\in L^2(\mathbb C^n).$ See \cite{T}, p.70. For $k\in\mathbb Z_+,$ the
Laguerre functions $\varphi_k^{n-1}$ is defined by
$\varphi_k^{n-1}(z)=L_k^{n-1}(\frac{1}{2}|z|^2)e^{-\frac{1}{4}|z|^2},$ where
\[L_k^{n-1}(x)=\sum_{j=0}^k(-1)^j\binom{k+n-1}{k-j}\frac{x^j}{j!},\] is the
Laguerre polynomial of degree $k$ and order $n-1.$

\begin{lemma} \label{lemma8}
Let $\tilde aP\in L^2(\mathbb C^n),$ where $\tilde a$ is a radial function and $P\in H_{p,q}$.
Then
\begin{eqnarray}\label{exp2}
\tilde aP\times\varphi_{j}^{n-1}(z)
&=&(2\pi)^{n}\left\langle\tilde a, \varphi_{k-p}^{n+p+q-1}\right\rangle P(z)\varphi_{k-p}^{n+p+q-1}(z)\nonumber\\
&=&(2\pi)^{-n}P(z)~\tilde a\times\varphi_{k-p}^{n+p+q-1}(z),
\end{eqnarray}
if ~$k\geq p$ and ~$0$ otherwise. The convolution in the right hand side is on
the space $~\mathbb C^{n+p+q}.$
\end{lemma}

\section{Sets of injectivity for the twisted spherical means}\label{section3}
In this section, we prove that the $X$-axis is a set of
injectivity for the TSM for a certain class of functions in
$L^2(\mathbb C).$ Then, we replicate the method to prove that
$X$-axis together with $Y$-axis is a set of injectivity for the
TSM for $L^q(\mathbb C).$ In the later case, we deduce a density
result for $L^p(\mathbb C),~2\leq p<\infty.$

\bigskip

Since Laguerre function $\varphi_k^{n-1}$ is an eigenfunction of the special
Hermite operator $A=-\Delta_z+\frac{1}{4}|z|^2,$ with eigenvalue $2k+n,$ the
projection $f\times\varphi_k^{n-1}$ is also an eigenfunction of $A$ with
eigenvalue $2k+n.$ As $A$ is an elliptic operator and eigenfunction of an
elliptic operator is real analytic \cite{N}, the projection $f\times\varphi_k^{n-1}$
must be a real analytic function on $\mathbb C^n.$ The real analyticity of
$f\times\varphi_k^{n-1}$ can also be understood by the fact that
$f\times\varphi_k^{n-1}$ can be extended to a holomorphic function
on $\mathbb C^{2n}.$ Therefore, any determining set for the real analytic functions
is a set of injectivity for the TSM on $L^q(\mathbb C^n)$ with $1\leq q\leq\infty.$
For details on determining sets for real analytic functions, see \cite{PS, RS}.
Next, we find an expansion for $f\times\varphi_k^0$ with help of Hecke-Bochner
identities for spectral projection.

\begin{proposition}\label{prop1}
Let $f\in L^2(\mathbb C).$ Then the real analytic expansion of $Q_k(z)=f\times\varphi_k^0(z)$
can be written as
\begin{equation}\label{exp5}
Q_k(z)=\sum_{p=0}^kC_{k-p}^{p0}z^p\varphi_{k-p}^{p}(z)+\sum_{q=0}^\infty C_{k}^{0q}\bar z^q\varphi_{k}^{q}(z).
\end{equation}
\end{proposition}

\begin{proof}
We know that
\[f(z)=\sum_{p=0}^\infty a^{p0}(|z|)z^p+ \sum_{q=0}^\infty a^{0q}(|z|)\bar z^q, \text{ for }z\in\mathbb C.\]
Since $f\in L^2(\mathbb C),$ using the Hecke-Bochner identity for the spectral projections
as in Lemma \ref{lemma8}, we can express $f\times\varphi_k^0(z)$ as
\[Q_k(z)=f\times\varphi_k^0(z)=
\sum_{p=0}^kC_{k-p}^{p0}z^p\varphi_{k-p}^{p}(z)+\sum_{q=0}^\infty C_{k}^{0q}\bar z^q\varphi_{k}^{q}(z),\]
where the series on the right-hand side converges to $Q_k$ in $L^2(\mathbb C).$
In order to show that the series converges uniformly on every compact set
$K\subseteq\mathbb C,$ it is enough to show that the series
\[h(z)=\sum_{q=k+1}^\infty C_{k}^{0q}\bar z^q\varphi_{k}^{q}(z),\]
converges uniformly on every ball $B_R(o)$ in $\mathbb C.$ Since $Q_k\in L^2(\mathbb C),$
it follows that $h\in L^2(\mathbb C)$ and
\[\|h\|^2_{L^2(\mathbb C)}=
\sum_{q=k+1}^\infty |C_{k}^{0q}|^2\|\bar z^q\varphi_{k}^{q}\|^2_{L^2(\mathbb C)}<\infty.\]
Since
\begin{eqnarray}\label{exp27}
\|\bar z^q\varphi_{k}^{q}\|^2_{L^2(\mathbb C)}
&=&\int_0^\infty\int_{S^1}|r\omega|^{2q}\left(\varphi_k^{q}\right)^2rdrd\omega\nonumber\\
&=&2\pi\int_0^\infty\left(\varphi_k^{q+1-1}\right)^2r^{2(q+1)-1}dr\nonumber\\
&=&2\pi~2^q\frac{(k+q)!}{k!}.
\end{eqnarray}
Therefore, the coefficients $C_k^{0q}$'s must satisfy an estimate of type
\begin{equation}\label{exp25}
|C_k^{0q}|\leq C\left(\frac{k!}{2^{q+1}(k+q)!}\right)^{\frac{1}{2}},
\end{equation}
where $C$ is a constant and independent of $q.$
Now, let $|z|\leq R.$ Then, we have
\begin{eqnarray*}
|h(z)|&\leq& e^{-\frac{1}{4}|z|^2}\sum_{q=k+1}^\infty|C_k^{0q}||z|^q
\left|\sum_{j=0}^k(-1)^j\binom{q+k}{k-j}\frac{({\frac{1}{2}|z|^2)}^j}{j!}\right|\\
&\leq& Ce^{-\frac{1}{4}|z|^2}\sum_{q=k+1}^\infty\left(\frac{k!}{2^{q+1}(k+q)!}\right)^{\frac{1}{2}}
|z|^q\frac{(q+k)!}{k!q!}\sum_{j=0}^k\frac{({\frac{1}{2}|z|^2)}^j}{j!}\\
&\leq& Ce^{-\frac{1}{4}|z|^2}
\sum_{q=k+1}^\infty\left(\frac{(q+k)!}{2^{q+1}k!q!}\right)^{\frac{1}{2}}\frac{|z|^q}{(q!)^{\frac{1}{2}}}e^{\frac{1}{2}|z|^2}\\
&\leq& Ce^{\frac{1}{4}R^2}
\sum_{q=k+1}^\infty\left(\frac{(q+k)!}{2^{q+1}k!q!}\right)^{\frac{1}{2}}\frac{R^q}{(q!)^{\frac{1}{2}}}<\infty.
\end{eqnarray*}
Thus the function $h$ is real analytic on $\mathbb C.$ That is, the right-hand side of
(\ref{exp5}) is a real analytic function which agreeing to the real analytic function $Q_k$
a.e. on $\mathbb C.$ Hence (\ref{exp5}) is a real analytic expansion of $Q_k.$
\end{proof}

We would like to call (\ref{exp5}) the \emph {Hecke-Bochner-Laguerre series} for the
spectral projections. We study this series carefully and use it to prove the most
striking results, Theorems (\ref{th1},~\ref{th4}) of this article.

\begin{theorem}\label{th2}
Let $f\in L^2(\mathbb C)$. Suppose $f\times\mu_r(x)=0, \forall~r>0$  and
$ \forall~x\in\mathbb R.$ Then $f\times\varphi_0^0\equiv f\times\varphi_1^0\equiv0$
on $\mathbb C.$
\end{theorem}

\begin{proof}
Since $f\times\mu_r(x)=0, ~\forall r>0$ and $ \forall x\in\mathbb R,$ by polar decomposition,
it follows that $Q_k(x)=f\times\varphi_k^0(x)=0, ~\forall x\in\mathbb R$ and $\forall k\in\mathbb Z_+.$
For $k=0,$ we have
\[Q_0(x)=C_{0}^{00}\varphi_{0}^{0}(x)+C_{0}^{00}\varphi_{0}^{0}(x)+C_{0}^{01}x\varphi_{0}^{1}(x)+
C_{0}^{02}x^2\varphi_{0}^{2}(x)+\cdots=0,
~\forall~x\in\mathbb R.\]
On equating the coefficients of $1,x,x^2,\ldots$ to zero, we get $C_0^{0q}=0, \forall q\geq0.$
Hence $Q_0\equiv0$ on $\mathbb C.$
From Equation (\ref{exp5}), we have
\[Q_1(x)=C_{1}^{00}\varphi_{1}^{0}(x)+C_{0}^{10}x\varphi_{0}^{1}(x)+
\sum_{q=0}^\infty C_{1}^{0q}x^q\varphi_{1}^{q}(x)=0,~\forall~x\in\mathbb R.\]
Using the argument $x\rightarrow-x,$ it follows that
\[2C_{1}^{00}\varphi_{1}^{0}(x)+\sum_{m=0}^\infty C_{1}^{0,2m}x^{2m}\varphi_{1}^{2m}(x)=0.\]
By equating coefficient of $1, x^2, x^4,\ldots,$  we get $C_1^{0,2m}=0,$ for $~m=0,1,2,\ldots.$
Hence the series of $Q_1(x)$ reduces to
\[Q_1(x)=C_{0}^{1,0}x\varphi_{0}^{1}(x)+\sum_{m=0}^\infty C_{1}^{0,2m+1}x^{2m+1}\varphi_{1}^{2m+1}(x)=0.\]
By canceling $e^{-\frac{1}{4}x^2}$ in the above series, we have
\[C_{0}^{1,0}x+\sum_{m=0}^\infty C_{1}^{0,2m+1}x^{2m+1}{\left(2m+2-\frac{1}{2}x^2\right)}=0.\]
On equation the coefficients of $x,x^3,x^5,\ldots,$ we get the following recursion relations
\begin{equation}\label{exp26}
C_0^{1,0}=-2C_1^{0,1} ~\text{and}~ C_1^{0,2m+1}=\frac{C_1^{0,1}}{2^{2m}(m+1)!};~\text{for}~ m=1,2,3,\ldots.
\end{equation}
Now, we can write $Q_1(z)=C_{0}^{1,0}z\varphi_{0}^{1}(z)+h(z),$
where the series
\[h(z)=\sum_{m=0}^\infty C_{1}^{0,2m+1}\bar z^{2m+1}\varphi_{1}^{2m+1}(z)\]
converges in $L^2(\mathbb C).$ We claim that all the coefficients
$C_1^{0,2m+1};~m=0,1,2,\ldots$ are zero. Here two cases arises.
If $h_1(z)$ has finitely many non-zero coefficients. Then $Q_1(z)$ is polynomial
times Gaussian and hence the condition $Q_1(x)=0, \forall x\in \mathbb R,$
by equating the coefficient of highest degree term to zero, implies that each
coefficient has to be zero. (Please see the proof of Theorem \ref{th1}.)
On the other hand, suppose infinitely many of these coefficients are non-zero. Then,
by the estimate (\ref{exp27}) and the recursion relations (\ref{exp26}), we have
\begin{eqnarray*}
\|h\|^2_{L^2(\mathbb C)}&=&\sum_{m=0}^{\infty}\left|C_1^{0,2m+1}\right|^2
\left\|\bar z^{2m+1}\varphi_1^{2m+1}\right\|^2_{L^2(\mathbb C)}\\
&=&2\pi\left|C_1^{0,1}\right|^2\sum_{m=0}^{\infty}\frac{2^{2m+1}(2m+2)!}{\left(2^{2m}(m+1)!\right)^2}\\
&=&4\pi\left|C_1^{0,1}\right|^2\sum_{m=0}^{\infty}\frac{(2m+2)!}{2^{2m}\left((m+1)!\right)^2}=\sum_{m=0}^\infty
 b_m=\infty.
\end{eqnarray*}
The series on the right-hand side diverges by Raabe's test. (See, \cite{K}, p.36).
Since
\[\lim_{m\rightarrow\infty}\left\{m\left(\frac{b_m}{b_{m+1}}-1\right)\right\}=-\frac{1}{2}<1.\]
This contradicts the fact that the series $h$ is $L^2(\mathbb C)$ summable.
Thus, we get $C_1^{0,2m+1}=0,$ for $m=0,1,2,\ldots.$ Hence, we conclude that $Q_1\equiv0.$
\end{proof}

\begin{remark}\label{rk3}
Under the same assumptions as in Theorem \ref{th2},  it would be interesting to know,
whether $Q_k\equiv0$ for $k\geq2.$ The argument used to show $Q_1\equiv0$ does not
seem to work in this case.
In another attempt, using the recursion relations $L_k^n=L_k^{n-1}+\cdots+L_0^{n-1},~
L_k^{n}-L_{k-1}^n=L_k^{n-1}$ and the result that $f\times\varphi_0^0=f\times\varphi_1^0=0,$
we can easily deduce that $f\times\varphi_2^0=f\times\varphi_2^1.$ But, we are not able to
conclude any thing more on account of the facts that $f\times\varphi_2^0$ is an eigenfunction
of the special Hermite operator $A$ and $\varphi_2^1=\varphi_2^0+\varphi_1^0+\varphi_0^0.$

\smallskip

However, we prove the following partial result that any line passing through the origin is a
set of injectivity for the TSM for a certain class of functions in $L^2(\mathbb C).$ Since
for any $\sigma\in U(n),$ we have $~f\times\mu_r(\sigma.z)=(\pi(\sigma)f)\times\mu_r(z).$ It
follows that a set $S\subset\mathbb C$ is a set of injectivity for the TSM if and only if for
each $\sigma\in U(n),$ the set $\sigma.S$  is a set of injectivity for the TSM.
In view of this, it is enough to prove that the $X$-axis is a set of injectivity for the TSM.
\end{remark}

\begin{theorem}\label{th1}
Let $f\in L^2(\mathbb C)$ and for each $k\in\mathbb Z_+$ the projection
 $e^{\frac{1}{4}|z|^2}f\times\varphi_k^0$ is a polynomial. Suppose
$f\times\mu_r(x)=0, \forall~r>0$  and  ~$\forall~x\in\mathbb R.$ Then $f=0$ a.e.
\end{theorem}

\begin{proof}
Since $f\in L^2(\mathbb C)$, by polar decomposition, the condition $f\times\mu_r(x)=0, \forall~r>0$
and $ \forall~x\in\mathbb R$ is equivalent to $f\times\varphi_k^0(x)=0, \forall~k\geq0$
and $ \forall~x\in\mathbb R.$
From Equation (\ref{exp5}), we have
\[f\times\varphi_k^0(z)=\sum_{p=0}^kC_{k-p}^{p0}z^p\varphi_{k-p}^{p}(z)+
\sum_{q=0}^\infty C_{k}^{0q}\bar z^q\varphi_{k}^{q}(z).\]
By the given exponential condition, we can write
$f\times\varphi_k^0(z)=P(z, \bar z)e^{-\frac{1}{4}|z|^2}.$
Let $z=te^{i\theta}.$ Then for each fixed $t,$ the function $f\times\varphi_k^0(te^{i\theta})$
is a trigonometric polynomial. Using the orthogonality of $e^{i n\theta}$
it follows that there exist $m=m(k)\in\mathbb Z_+$ such that
\begin{equation}\label{exp7}
f\times\varphi_k^0(z)=\sum_{p=0}^kC_{k-p}^{p0}z^p\varphi_{k-p}^{p}(z)+
\sum_{q=0}^mC_{k}^{0q}\bar z^q\varphi_{k}^{q}(z).
\end{equation}
 Therefore, for each $k\in\mathbb Z_+,$ we have
\[\sum_{p=0}^kC_{k-p}^{p0}x^p\varphi_{k-p}^{p}(x)+
\sum_{q=0}^mC_{k}^{0q}x^q\varphi_{k}^{q}(x)=0, \forall~x\in\mathbb R.\]
The constant term in the above expansion $2C_k^{00}=0,$ hence
\[\sum_{p=1}^kC_{k-p}^{p0}x^p\varphi_{k-p}^{p}(x)+
\sum_{q=1}^mC_{k}^{0q}x^q\varphi_{k}^{q}(x)=0, \forall~x\in\mathbb R.\]
On equating the coefficient of the highest degree term $x^{m+2k}$ to zero,
we get $C_k^{0m}=0.$ Similarly, continuing this argument up to $x^{2k},$
we obtained $C_k^{0m}=C_k^{0(m-1)}=\cdots=C_k^{01}=0.$ Then equate the
coefficients of $x,x^2,\ldots,x^{2k-1}$ to zero, we find
$C_{k-1}^{10}=C_{k-2}^{20}=\cdots=C_0^{p0}=0.$ Thus $f\times\varphi_k^{0}\equiv0, \forall k\geq0.$
Hence $f=0$ a.e. This completes the proof.
\end{proof}

\begin{remark}\label{rk5}
$(a).$ Since $f\times\varphi_k$ is real analytic and the zero set of a real analytic function
is isolated, in this case, we only need the centres to be a sequence in $\mathbb R$ having a
limit point. It is clear from (\ref{exp7}) that any curve $\gamma:=\{(\gamma_1(t),\gamma_2(t)): t\in\mathbb R\},$
where $\gamma_j,~j=1,2$ are polynomials is also a set of injectivity for the TSM.
A natural question is that $\gamma:=(\gamma_1,\gamma_2)$ with $\gamma_j$'s are real
analytic is a set of injectivity for the TSM. We believe that this will help in
characterizing non-injectivity sets for TSM.

\bigskip

$(b).$ Let us consider the functions
\[f_m(z)=\sum_{p=0}^\infty a^{p0}(|z|)z^p+ \sum_{q=0}^m a^{0q}(|z|)\bar z^q .\]
Then by the similar argument as in the proof of Theorem \ref{th2}, we can easily
deduce that $e^{\frac{1}{4}|z|^2}f_m\times\varphi_k^0$ is a polynomial. In fact
these are the only functions for which $e^{\frac{1}{4}|z|^2}f\times\varphi_k^0$
is a polynomial. This can be seen from Equation (\ref{exp7}), when we reverse the
process using the Hecke-Bochner identities. Thus the space considered is just not
empty, it includes all the sequence $(f_m)$ which converges to
\[f(z)=\sum_{p=0}^\infty a^{p0}(|z|)z^p+ \sum_{q=0}^\infty a^{0q}(|z|)\bar z^q \]
in $L^2(\mathbb C).$

We believe that the proof of Theorem \ref{th1}, without
exponential condition on spectral projection would need a finer argument and
hence we prefer to return to this question later.
\end{remark}

Next, we prove the stronger result that $X$-axis together with $Y$-axis is a set of
injectivity for the TSM for any function in $L^q(\mathbb C).$

For $\eta\in\mathbb C,$ define the left twisted translate by
$$\tau_\eta f(\xi)=f(\xi-\eta)e^{\frac{i}{2}\text{Im}(\eta.\bar\xi)}.$$
Then $\tau_{\eta}(f\times\mu_r)=\tau_{\eta}f\times\mu_r.$ Let $S$ be a set of
injectivity for the TSM on $L^q(\mathbb C).$ Suppose $f\times\mu_r(z-\eta)=0,\forall r>0$
and $\forall z\in S.$ Then
\[\tau_{\eta}f\times\mu_r(z)=\tau_{\eta}(f\times\mu_r)(z)=
e^{\frac{i}{2}\text{Im}(\eta.\bar{z})}f\times\mu_r(z-\eta)=0,\]
for all $r>0$ and $\forall z\in S.$  Since the space $L^q(\mathbb C)$ is twisted
translations invariant, it follows that a set $S\subset\mathbb C$ is set of
injectivity for the TSM if and only if for each $\eta\in\mathbb C,$ the set $S-\eta$
is a set of injectivity for the TSM. That is, the Euclidean translate of the set $S$
is also a set of injectivity for the TSM on $L^q(\mathbb C).$ By rotation and translation,
it is obvious that any two perpendicular lines can be set of injectivity for the TSM,
provided $X$-axis and $Y$-axis together is a set of injectivity for the TSM.

\begin{theorem}\label{th4}
Let $f\in L^q(\mathbb C),$ for $1\leq q\leq2.$ Suppose $f\times\mu_r(x)=f\times\mu_r(ix)=0, \forall~r>0$
and $ \forall~x\in\mathbb R.$ Then $f=0$ a.e.
\end{theorem}

Let $f\in L^q(\mathbb C).$ Then by convolving $f$ with a right and radial compactly supported smooth
approximate identity, we can assume $f\in L^2(\mathbb C).$
Let us consider the Hecke-Bochner-Laguerre series
\begin{equation}\label{exp6}
Q_k(z)=\sum_{p=1}^k\left(C_{k-p}^{p0}z^p\varphi_{k-p}^{p}(z)+C_{k}^{0p}\bar z^p\varphi_{k}^{p}(z)\right)
+\sum_{p=k+1}^\infty C_{k}^{0p}\bar z^p\varphi_{k}^{p}(z).
\end{equation}
The proof is now based on symmetries and then cancelations. We decompose
the above series into four (disjoint) series, each of which after equating
its coefficients to zero, gives a system of solvable recursion relations. Using these
recursion relations together with some basic properties of Laguerre polynomials,
we show that all the coefficients appeared in the series (\ref{exp6}) are zero.

\smallskip

Let $\mathbb E_+$ and $\mathbb O_+$ denote the sets of even and odd positive
integers respectively. Let $E_k=\mathbb E_+\cap\{1,2,\ldots, k\},~
F_k=\mathbb E_+\cap\{k+1,k+2,\ldots\},~G_k=\mathbb O_+\cap\{1,2,\ldots, k\}$ and
$H_k=\mathbb O_+\cap\{k+1,k+2,\ldots\}.$ Then, we can decompose the above series
as $Q_k(z)=U_k(z)+V_k(z^2),$ where
\[U_k(z)=\sum_{p\in G_k}\left(C_{k-p}^{p0}z^p\varphi_{k-p}^{p}(z)+C_{k}^{0p}\bar z^p\varphi_{k}^{p}(z)\right)
+\sum_{p\in H_k} C_{k}^{0p}\bar z^p\varphi_{k}^{p}\]
and
\[V_k(z^2)=\sum_{p\in E_k}\left(C_{k-p}^{p0}z^p\varphi_{k-p}^{p}(z)+C_{k}^{0p}\bar z^p\varphi_{k}^{p}(z)\right)
+\sum_{p\in F_k} C_{k}^{0p}\bar z^p\varphi_{k}^{p}.\]
We shall call $U_k$ and $V_k$ as odd and even series respectively.
For $x\in\mathbb R,$ it is given that $U_k(x)+V_k(x^2)=0.$  Using
the argument $x\rightarrow -x,$ we have $U_k(x)=V_k(x^2)=0.$
Similarly, $U_k(ix)+V_k((ix)^2)=0,$ implies $U_k(ix)=V_k((ix)^2)=0.$
Indeed, we can put these conditions as follow.

\begin{description}
  \item[(A)] $U_k(x)=U_k(ix)=0\text{ and }$
  \item[(B)] $V_k(x^2)=V_k((ix)^2)=0.$
\end{description}
In order to prove Theorem \ref{th4}, we prove that $Q_k\equiv0, \forall k\geq0.$
Now, we divide the proof into two parts: $0\leq k\leq3$ and $k\geq4.$

\begin{lemma}\label{lemma9}
Let $f\in L^2(\mathbb C)$ and $0\leq k\leq3.$ Suppose $Q_k(x)=Q_k(ix)=0, ~\forall x\in\mathbb R.$
Then $Q_k\equiv0.$
\end{lemma}

\begin{proof}
Since, we have shown in Theorem \ref{th2} that $Q_0\equiv
Q_1\equiv0,$ we only need to prove $Q_k\equiv0$ for $k=2,3.$ For
$k=2,$ by conditions (A), we get a pair of equations
\[x\left(C_1^{10}\varphi_1^1(x)+C_2^{01}\varphi_2^1(x)\right)+
\sum_{m=2}^{\infty}C_2^{0,2m-1}x^{2m-1}\varphi_2^{2m-1}(x)=0, \]
\[x\left(C_1^{10}\varphi_1^1(x)-C_2^{01}\varphi_2^1(x)\right)+
\sum_{m=2}^{\infty}(-1)^mC_2^{0,2m-1}x^{2m-1}\varphi_2^{2m-1}(x)=0.\]
On adding these two equations, we have
\[xC_1^{10}\varphi_1^1(x)+\sum_{m=2}^{\infty}C_2^{0,4m-5}x^{4m-5}\varphi_2^{4m-5}(x)=0, \]
By equating the coefficients of $x, x^3, x^7,\ldots$ to zero, we get
$C_1^{10}=0$ and $C_2^{0,4m-5}=0,$ for $m=2,3,\ldots.$
Similarly by subtracting and then equating the coefficients of
$x,x^5, x^9,\ldots$ to zero, we obtain $C_2^{01}=0$ and $C_2^{0,4m-7}=0,$
for $m=3,4,\ldots.$ Hence $U_2\equiv0.$
By conditions (B), we have
\[x^2\left(C_0^{20}\varphi_0^2+C_2^{02}\varphi_2^2\right)+\sum_{m=2}^{\infty}C_2^{0,2m}x^{2m}\varphi_2^{2m}(x)=0, \]
\[-x^2\left(C_0^{20}\varphi_0^2+C_2^{02}\varphi_2^2\right)+\sum_{m=2}^{\infty}(-1)^mC_2^{0,2m}x^{2m}\varphi_2^{2m}(x)=0.\]
In a quite similar way, we find $V_2\equiv0$ and hence $Q_2\equiv0.$
Here, by adding and subtracting, we get the coefficients to be more disjoint,
which is the only difficulty we need to resolve. The method used in this
case will be repeated for $k\geq3.$

For $k=3,$ by conditions (A), we have
\[x\left(C_2^{10}\varphi_2^1(x)+C_3^{01}\varphi_3^1(x)\right)+\sum_{m=3}^{\infty}C_3^{0,2m-1}x^{2m-1}\varphi_3^{2m-1}(x)=0,\]
\[x\left(C_2^{10}\varphi_2^1(x)-C_3^{01}\varphi_3^1(x)\right)+\sum_{m=3}^{\infty}(-1)^mC_3^{0,2m-1}x^{2m-1}\varphi_3^{2m-1}(x)=0.\]
As very similar to above, the pair of equations implies that $U_3\equiv0.$
By conditions (B), we have
\[x^2\left(C_1^{20}\varphi_1^2(x)+C_3^{02}\varphi_3^2(x)\right)+\sum_{m=2}^{\infty}C_3^{0,2m}x^{2m}\varphi_3^{2m}(x)=0,\]
\[-x^2\left(C_1^{20}\varphi_1^2(x)+C_3^{02}\varphi_3^2(x)\right)+\sum_{m=2}^{\infty}(-1)^mC_3^{0,2m}x^{2m}\varphi_3^{2m}(x)=0.\]
This shows that $V_3\equiv0$  and hence $Q_3\equiv0.$
\end{proof}
Next, we prove the following lemma for the case $k\geq4,$ which
completes the proof of Theorem \ref{th4}.

\begin{lemma}\label{lemma10}
Let $f\in L^2(\mathbb C)$ and $~k\geq4.$ Suppose $Q_k(x)=Q_k(ix)=0,~\forall x\in\mathbb R.$
Then $Q_k\equiv0.$
\end{lemma}

\begin{proof}
First, we show that the odd series $U_k\equiv0, \forall k\geq4.$ By conditions (A),
we can write
\[\sum_{p\in G_k}x^p\left(C_{k-p}^{p0}\varphi_{k-p}^{p}(x)+C_{k}^{0p}\varphi_{k}^{p}(x)\right)+
\sum_{p\in H_k} C_{k}^{0p}x^p\varphi_{k}^{p}(x)=0\]
\[\sum_{p\in G_k}(-1)^{\frac{p-1}{2}}x^p\left(C_{k-p}^{p0}\varphi_{k-p}^{p}(x)-C_{k}^{0p}\varphi_{k}^{p}(x)\right)+
\sum_{p\in H_k}(-1)^{\frac{p+1}{2}}C_{k}^{0p}x^p\varphi_{k}^{p}(x)=0.\]
By adding and subtracting, we get two series in which there are no common coefficients.
On equating  the coefficients of $x, x^3, x^5,\ldots,$ in both the new series to zero, we get all
the coefficients of odd series $U_k$ are zero. Hence $U_k\equiv0, \forall k\geq4.$ Since $i^4=1,$
it shows that there must occur some change in the pattern of the even series $V_k$ for $k\geq4.$
For instance consider $V_4.$
By condition (B), we have
\[x^2\left(C_2^{20}\varphi_2^2(x)+C_4^{02}\varphi_4^2(x)\right)+x^4\left(C_0^{40}\varphi_0^4(x)+C_4^{04}\varphi_4^4(x)\right)+\]
\[\sum_{m=3}^{\infty}C_4^{0,2m}x^{2m}\varphi_4^{2m}(x)=0,\]
\[-x^2\left(C_2^{20}\varphi_2^2(x)+C_4^{02}\varphi_4^2(x)\right)+x^4\left(C_0^{40}\varphi_0^4(x)+C_4^{04}\varphi_4^4(x)\right)+\]
\[\sum_{m=3}^{\infty}(-1)^mC_4^{0,2m}x^{2m}\varphi_4^{2m}(x)=0.\]
On adding and subtracting, we get the two series
\begin{equation}\label{exp14}
 x^4\left(C_0^{40}\varphi_0^4(x)+C_4^{04}\varphi_4^4(x)\right)+\sum_{m=4}^{\infty}C_4^{0,4(m-2)}x^{2m}\varphi_4^{4(m-2)}(x)=0,
\end{equation}
\begin{equation}\label{exp15}
x^2\left(C_2^{20}\varphi_2^2(x)+C_4^{02}\varphi_4^2(x)\right)+\sum_{m=3}^{\infty}C_4^{0,4m-6}x^{4m-6}\varphi_4^{4m-6}(x)=0.
\end{equation}
By equating the coefficients of $x^4, x^6, x^8,\ldots,$ in Equation (\ref{exp14}) to zero,
we get, all the coefficients in (\ref{exp14}) are zero.
By canceling $e^{-\frac{1}{4}x^2}$ in series (\ref{exp15}) and using $x^2\rightarrow 2x,$ we have
\[2x\left(C_2^{20}L_2^2(x)+C_4^{02}L_4^2(x)\right)+\sum_{m=3}^{\infty}C_4^{0,4m-6}(2x)^{2m-3}L_4^{4m-6}(x)=0.\]
On equating the coefficients of $x$ and $x^2$ to zero, we get
\begin{equation} \label{exp16}
 \left(
  \begin{array}{cccc}
  L_{2}^2(0)  &  L_4^2(0) \\
  (L_{2}^2)^{'}(0)  &  (L_4^2)^{'}(0) \\
  \end{array}
\right)\left(
\begin{array}{c}
           C_{2}^{20}\\
           C_4^{02} \\
\end{array}
\right)=
\left(
  \begin{array}{cccc}
 6  &    6 \\
-4  &   -20 \\
  \end{array}
\right)\left(
\begin{array}{c}
           C_{2}^{20}\\
           C_4^{02} \\
\end{array}
\right)=0.
\end{equation}
Thus $C_{2}^{20}=C_4^{02}=0$ and hence we find $V_4\equiv0.$ Equivalently,
we can use onwards to write: on equating the coefficients of $x^2$ and $x^4$
in Equation (\ref{exp15}) to zero, we get Equations (\ref{exp16}).
Now, it only remains to show that $V_k\equiv0, \forall k\geq5.$
In this case, by conditions (B), we have
\[\sum_{p\in E_k}x^p\left(C_{k-p}^{p0}\varphi_{k-p}^{p}+C_{k}^{0p}\varphi_{k}^{p}\right)+\sum_{p\in F_k} C_{k}^{0p}x^p\varphi_{k}^{p}=0,\]
\[\sum_{p\in E_k}(-1)^{\frac{p}{2}}x^p\left(C_{k-p}^{p0}\varphi_{k-p}^{p}+C_{k}^{0p}\varphi_{k}^{p}\right)
+\sum_{p\in F_k}(-1)^{\frac{p}{2}}C_{k}^{0p}x^p\varphi_{k}^{p}=0.\]
We further require a partition of the set $\mathbb E=A_1\cup A_2,$ where $A_1=\{4t-2: t\in\mathbb N\}$ and $A_2=\{4t: t\in \mathbb N\}.$
By adding and subtracting, we will get the following pair of series having brackets.
\begin{equation}\label{exp17}
\sum_{p\in{ E_k\cap A_2}}x^p\left(C_{k-p}^{p0}\varphi_{k-p}^{p}+C_{k}^{0p}\varphi_{k}^{p}\right)+\sum_{p\in {F_k\cap A_2}} C_{k}^{0p}x^p\varphi_{k}^{p}=0,
 \end{equation}
\begin{equation}\label{exp18}
\sum_{p\in{ E_k\cap A_1}}x^p\left(C_{k-p}^{p0}\varphi_{k-p}^{p}+C_{k}^{0p}\varphi_{k}^{p}\right)+\sum_{p\in {F_k\cap A_1}} C_{k}^{0p}x^p\varphi_{k}^{p}=0,
 \end{equation}
Since the matrix
\[\left(
  \begin{array}{cccc}
  L_{k-p}^p(0)  &  L_k^p(0) \\
  (L_{k-p}^p)^{'}(0)  &  (L_k^p)^{'}(0) \\
  \end{array}
\right)\]
is non-singular and $p\in A_2,$ the brackets in (\ref{exp17}) is not a problem.
Hence on equating the coefficients of $x^2, x^4, x^6,\ldots$ to zero in (\ref{exp17}),
it follows that  all the coefficients in (\ref{exp17}) are zero. Similarly, in (\ref{exp18}),
as $p$ in $A_1,$  it also follows that all the coefficients in (\ref{exp18}) are zero.
Thus we find $V_k\equiv0, \forall k\geq5$ and hence $Q_k\equiv0,\forall k\geq0.$
This completes the proof.
\end{proof}

\begin{remark} \label{rk1}
$(a).$ We can also prove the general case by calculating case-wise, when $k=2m, 2m+1$
and $p=4t-2, 4t,$ but it would only make the calculation to be more complicated.
In another attempt, to get a more transparent proof of Theorem \ref{th4}, keep applying
the right invariant operator
$\tilde A=\frac{\partial}{\partial z}+\frac{1}{4}\bar z$ to $Q_k(z).$ Then a straightforward
calculation shows that
\[\tilde A^pQ_k(0)=p!\varphi_{k-p}^p(0)C_{k-p}^{p0},\]
if ~$p\leq k$ and ~$0$ otherwise.
By the condition
$Q_k(x)=Q_k(ix)=0, \forall x\in\mathbb R,$ it follows that $\tilde A Q_k(0)=0.$ This implies
$C_{k-1}^{10}=0$ and in turn $C_k^{01}=0.$ In view of this, we can immediately conclude that
$Q_1\equiv0.$ But for $k\geq2,$ the conditions on $Q_k$ do not imply $\tilde A^pQ_k(0)=0,$
when $p\geq2,$ otherwise this would leads to a more transparent proof
of Theorem \ref{th4}.

\smallskip

$(b).$  In Theorem \ref{th4}, we have shown that any two lines having angle $\pi/2$ is a
set of injectivity for the TSM on $\mathbb C.$ However, the question that  any two lines
having positive angle less than $\pi/2$ can be a set of injectivity for the TSM on
$\mathbb C$ is still unanswered.

\smallskip

$(c).$ Consider Coxeter system of $N$-lines
$\Sigma_N=\cup_{l=0}^{N-1}\{te^{i\theta_l}: \theta_l=\frac{\pi l}{N}, t\in\mathbb R\}.$
Suppose $\theta_l=\pi/2.$ Then $l=N/2.$ By Theorem \ref{th4}, it follows that any
Coxeter system of even number of lines is also a set of injectivity for the TSM on $L^q(\mathbb C).$
\end{remark}
Next, we set to describe the problem for any Coxeter system of odd lines.
Let $1,\omega, \omega^2,\ldots,\omega^{N-1}$ be the $N$ roots of unity. Then
$\Sigma_N=\cup_{l=0}^{N-1}\{\omega^lx: x\in\mathbb R\}.$ We have formulated
this problem in the following way.

\bigskip

\noindent{\bf Conjecture.}
Let $f\in L^2(\mathbb C).$ Suppose $f\times\mu_r(\omega^lx)=0, \forall r>0,$ and $l=0,1,\ldots,N-1$
and $\forall~x\in\mathbb R.$ Then $f=0$ a.e.

\bigskip

We would like to produce some evidence about the feasibility of this problem.
In this case, we also get a system of solvable recursion relations by decomposing
the series using symmetries, however those recursion relations for higher values of $k,$
gives rise to a higher order square matrix, which is needed to show non-singular.
This is the only difficulty in getting a solution to this problem.
Since $Q_k(\omega^lx)=0, \forall l;~l=0,1,\ldots,N-1$ and $\forall x\in\mathbb R,$
as similar to Theorem \ref{th4}, we will have the following conditions.

\begin{description}
 \item[(A)] $U_k(\omega^lx)=0, \forall ~l;~l=0,1,\ldots,N-1 \text{ and }$
 \item[(B)] $V_k((\omega^lx)^2)=0, \forall ~l;~l=0,1,\ldots,N-1.$
 \end{description}
For $N=3,$ by a simple argument that find all odd positive integers which are not divisible by $3$,
we can find a partition of the set of natural numbers as $\mathbb N=\cup_{i=0}^2A_i,$ where
$A_o=\{1,2,3,4\}, ~A_1=\{6t-1, 6t: t\in\mathbb N\}$ and $A_2=\{6t+1,6t+2,6t+3,6t+4: t\in\mathbb N\}.$

For $k\in A_o,$ the conditions (A) together with the facts $1+\omega+\omega^2=0$ and $\omega^3=1,$
implies that
\[C_k^{03}x^3\varphi_k^3+C_k^{09}x^9\varphi_k^9+\cdots=0.\]
On equating the coefficient of $x^3,x^9,\ldots$ to zero, we get $C_k^{03}=C_k^{09}=\cdots=0.$
Hence the odd series $U_k$ reduces to
\begin{equation} \label{exp9}
x\left(C_{k-1}^{10}\varphi_{k-1}^1+C_k^{01}\varphi_k^1\right)+C_k^{05}x^5\varphi_k^5+C_k^{07}x^7\varphi_k^7+\cdots=0.
\end{equation}
On equating the coefficients of $x$ and $x^3$ in Equation (\ref{exp9}) to zero, we get
\begin{equation} \label{exp10}
 \left(
  \begin{array}{cccc}
  L_{k-1}^1(0)  &  L_k^1(0) \\
  (L_{k-1}^1)^{'}(0)  &  (L_k^1)^{'}(0) \\
  \end{array}
\right)\left(
\begin{array}{c}
           C_{k-1}^{10}\\
           C_k^{01} \\
\end{array}
\right)=
\left(
  \begin{array}{cccc}
 k &    k+1 \\
\frac{-k(k-1)}{2}  &   \frac{-k}{2} \\
  \end{array}
\right)\left(
\begin{array}{c}
           C_{k-1}^{10}\\
           C_k^{01} \\
\end{array}
\right)=0.
\end{equation}
Thus $C_{k-1}^{10}=C_k^{01}=0$ and hence we find  $U_k\equiv0,\forall k\in A_o.$
For $k\in\{1,2,3\},$ by condition (B) and the facts $1+\omega+\omega^2=0$ and $w^3=1,$ we have
\[C_k^{06}x^6\varphi_k^6+C_k^{0,12}x^{12}\varphi_k^{12}+\cdots=0.\]
This shows that $C_2^{06}=C_2^{0,12}=\cdots=0,$ and in turn the even series $V_k$ reduces to
\[x^2\left(C_{k-2}^{20}\varphi_{k-2}^k+C_k^{02}\varphi_k^2\right)+C_k^{04}x^4\varphi_k^4+C_k^{08}x^8\varphi_k^8+\cdots=0.\]
Let $F_k^n=L_k^n(0).$ On equating the coefficients of $x^2, x^4$ and $x^6$
to zero, we get
\[\left(
  \begin{array}{cccc}
 F_{k-2}^2  &   F_k^2       &     0 \\
 F_{0}^2    &  -F_{k-1}^3   &  F_k^4\\
   0        &   F_{k-2}^4   & -F_{k-1}^5\\
\end{array}
\right)\left(
\begin{array}{c}
           C_{k-2}^{20}\\
           C_k^{02} \\
           C_k^{04}\\
\end{array}
\right)=0.\]
This implies $C_{k-2}^{20}=C_k^{02}=C_k^{04}=0.$ Thus, it follows that $V_k\equiv0$ and
hence $Q_k\equiv0,$ for $k\in\{0,1,2,3\}.$ This gives a strong evidence about the existence
of the Problem. We would also like to focus on to the proof, for higher values of $k.$
Let $k\in A_1\cup A_2.$ Then by the conditions (A) together with the facts
$1+\omega+\omega^2=0$ and $\omega^3=1,$ the odd Series $U_k$ reduces to
\begin{eqnarray*}\label{exp11}
 &&x\left(C_{k-1}^{10}\varphi_{k-1}^1+C_k^{01}\varphi_k^1\right)+ x^5\left(C_{k-5}^{50}\varphi_{k-5}^5+C_k^{05}\varphi_k^5\right)+
x^7\left(C_{k-7}^{70}\varphi_{k-7}^7+C_k^{07}\varphi_k^7\right)\nonumber\\
 && +\cdots+ x^r\left(C_{k-r}^{r0}\varphi_{k-r}^r+C_k^{0r}\varphi_k^r\right)+\cdots+ C_k^{0j}x^j\varphi_k^j+\cdots=0.
\end{eqnarray*}
On equating the coefficients of $x$ and $x^3$ to zero, we get $C_{k-1}^{10}=C_k^{01}=0.$
Thus
\begin{eqnarray}\label{exp12}
&& x^5\left(C_{k-5}^{50}\varphi_{k-5}^5+C_k^{05}\varphi_k^5\right)+x^7\left(C_{k-7}^{70}\varphi_{k-7}^7+C_k^{07}\varphi_k^7\right)+\cdots \nonumber\\
&& +x^r\left(C_{k-r}^{r0}\varphi_{k-r}^r+C_k^{0r}\varphi_k^r\right) +\cdots+ C_k^{0j}x^j\varphi_k^j+\cdots=0.
\end{eqnarray}
The main problem here is to remove the brackets (mixed term) in this series. Then, it is easy to show that
rest of coefficients are zero. To remove the brackets, we need to identify the indices  $r, j,$ appeared in
(\ref{exp12}) and the number of equations $m$ required. By  Equation (\ref{exp12}),
 we can write the following table.

\bigskip

\begin{center}
\begin{tabular}{|l|l|l|l|l|}
\hline
$k$   &  $r$   & $j$ & $m=\frac{j+1}{2}-2$ & $k-r$ \\
\hline
$k\in A_1$ & $6t-1$ & $2k-1$ & $k-2$ & $0,1$  \\
\hline
$k\in A_2$ & $6t+1$ & $2k+1$ & $k-1$ & $0,1,2,3$  \\
\hline
\end{tabular}
\end{center}
\bigskip

Let $k\in A_1.$  Equate the coefficients of $x^5, x^7, \ldots, x^j$ to zero.
In order to show that coefficients in the brackets are zero, we need to show
that the following matrices are non-singular. For $k=5,$ we have the matrix
\begin{equation*} \label{exp13}
 \left(
  \begin{array}{ccccccccccc}
   F_{0}^5   &  F_5^5        &  0      \\
    0        & -F_{4}^6    &  F_{5}^7  \\
    0        &  F_{3}^7    & -F_{4}^8  \\
\end{array}
\right),
\end{equation*}
which is non-singular. For higher values of $k,$ we get the higher order matrices
which should be non-singular. This is the only difficulty in the above arguments.
Hence, we leave this problem open for future research.

\begin{remark} \label{rk7}
In the case, when $N=1$ (without exponential decay),  we can not obtain the
$m\times m$ matrices. However for $N\geq3,$ we have the $m\times m$  matrices
which is needed to be non-singular.
\end{remark}

From Remark \ref{rk1}(b), it is clear that the set $\Sigma_{2N}$
is a set of injectivity for the TSM for $L^q(\mathbb C),$ with
$1\leq q\leq2.$ As a dual problem, it is natural to ask that
$\Sigma_{2N}$ is a set of density for $L^p(\mathbb C),$ for $2\leq
 p<\infty.$ The following result would emerge as the first result
about the sets of density in terms of the TSM. Let
$C_c^\sharp(\mathbb C)$ denote the space of radial compactly
supported continuous functions on $\mathbb C.$ Let
$\tau_zf(w)=f(z-w)e^{\frac{i}{2}\text{Im}(z.\bar w)}.$

\begin{proposition}\label{prop3}
The subspace $\mathscr F(\Sigma_{2N})
={\emph{Span}}\left\{\tau_zf: z\in\Sigma_{2N}, f\in C_c^\sharp(\mathbb C) \right\}$
is dense in $L^p(\mathbb C),$ for $2\leq p<\infty.$
\end{proposition}

\begin{proof}
Let $\frac{1}{p}+\frac{1}{q}=1.$ Then $1\leq q\leq2.$ By Hahn-Banach theorem, it is
enough to show that $\mathscr F(\Sigma_{2N})^\bot=\{0\}.$ Let $g\in L^q(\mathbb C)$
be such that
\[\int_{\mathbb C}\tau_zf(w)g(w)dw=0, z\in\Sigma_{2N}, ~\forall f\in C_c^\sharp(\mathbb C).\]
That is, \[\overline{\bar g\times\bar f}(z)= f\times g(z)=0.\]
Let the support of $f$ be contained in $[0,t].$ Then by passing to the polar decomposition, we get
\[\int_{r=0}^t\bar g\times\mu_r(z)\bar f(r)r^{2n-1}dr=0.\]
By differentiating the above equation, it follows that $\bar g\times\mu_t(z)=0, \forall t>0$
and $\forall z\in\Sigma_{2N}.$ Thus by Theorem \ref{th4}, we conclude that $g=0$ a.e. on $\mathbb C^n.$
\end{proof}

\section{Discussion on sets of injectivity in higher dimension}\label{section4}
More generally,  similar to the work of Agranovsky and Quinto \cite{AQ}, let
$f\in L^1_{\loc}(\mathbb C^n)$ and write $S(f)=\{z\in\mathbb C^n: f\times\mu_r(z)=0, \forall~r>0\}.$
Our main problem is to describe the complete geometrical structure of $S(f)$ that would ensure which
``sets" are sets of injectivity for the TSM. There is one such result.
\begin{lemma} \label{lemma7}
 Let $f\in L^p(\mathbb C^n)\cap C(\mathbb C^n),$ for $ 1\leq p\leq\infty.$
 Then \[S(f)=\bigcap_{k=0}^\infty Q_k^{-1}(0).\]
\end{lemma}
\begin{proof}
Let $z\in S(f).$  Then by polar decomposition, it follows that $Q_k(z)=0, \forall k\geq0.$
Conversely, let $Q_k(z)=0, \forall k\geq0.$ Then
\[\int_{r=0}^\infty f\times\mu_r(z)\varphi_k^{n-1}(r)r^{2n-1}dr=0, \forall k\geq0.\]
Since the set $\{\varphi_k^{n-1}: k=0,1,2,\ldots\}$ is an orthonormal set for
$L^2\left(\mathbb R_+, r^{2n-1}dr\right)$ and $f\times\mu_r(z)$ is continuous in $r,$
it follows that $f\times\mu_r(z)=0, \forall r>0$ and hence $z\in S(f).$
\end{proof}

Next, we find out $S(f)$ for the type function
$f(z)=\tilde{a}(|z|)P(z)$ on $\mathbb C^n.$ For this, we need the
following result of Filaseta and Lam \cite{FL}, about the
irreducibility of Laguerre polynomials. Define the Laguerre
polynomials by
\[L^\alpha_k(x)=\sum_{i=0}^k(-1)^i\binom{\alpha+k}{k-i}\frac{x^i}{i!},\]
$\text{ where } k\in \mathbb Z_{+} \text{ and } \alpha\in\mathbb C.$
\begin{theorem}\label{th6} \emph{\cite{FL}} Let $\alpha$ be a rational number, which
is not a negative integer. Then for all but finitely many $k\in\mathbb Z_+$, the polynomial
$L_k^\alpha(x)$ is irreducible over the rationals.
\end{theorem}
Using Theorem \ref{th6}, we have obtained the following corollary about the zeros of Laguerre
polynomials.
\begin{corollary}\label{cor1}
Let $k\in\mathbb Z_{+}.$ Then for all but finitely many $k$, the Laguerre polynomials
$L^{n-1}_k(x)$'s have distinct zeros over the reals.
\end{corollary}
\begin{proof}
 By Theorem \ref{th6}, there exists $k_o\in\mathbb Z_+$ such that $L_k^{n-1}$'s are
irreducible over $\mathbb Q$ whenever  $k\geq k_o.$ Therefore, we can find polynomials
$P_1, P_2\in\mathbb Q[x]$ such that $P_1L_{k_1}^{n-1}+ P_2L_{k_2}^{n-1}=1,$ over $\mathbb Q$
with $k_1,k_2\geq k_o.$ Since this identity continue to hold on $\mathbb R,$ it follows that
$L_{k_1}^{n-1}$ and $L_{k_2}^{n-1}$ have no common zero over $\mathbb R.$
\end{proof}

\begin{proposition}\label{prop4}
Let $f$ be a non-zero type function $f=\tilde{a}P\in L^2(\mathbb
C^n)$, where $P\in H_{p,q}$. Then $S(f)=P^{-1}(0)\cup F,$ where $F$
is a finite union of spheres in $\mathbb C^n.$
\end{proposition}
\begin{proof}
Since $f\not\equiv0,$ there exists at least some $k\in\mathbb Z_+$ for which
$Q_k^{-1}(0)\neq\mathbb C^n.$ Therefore,
\[Q_k^{-1}(0)=P^{-1}(0)\cup\left(\varphi_{k-p}^{n+p+q-1}\right)^{-1}(0),\]
for some $k\in\mathbb Z_+.$ Hence $S(f)=P^{-1}(0)\cup F.$
\end{proof}

\begin{proposition}\label{pro2}
Let $f=\tilde{a}P\in L^2(\mathbb C^n)$, where $P\in H_{p,q}$. Suppose $Q_k$ is not
identically zero on $\mathbb C^n,$ for all but finitely many $k.$ Then $S(f)=P^{-1}(0).$
\end{proposition}
\begin{proof}
Since $Q_k(z)=f\times\varphi_k^{n-1}.$ Then by Lemma \ref{lemma7},
we have $S(f)=\cap_{k=0}^\infty Q_k^{-1}(0).$ By Hecke-Bochner
identity, we can write
\[Q_k(z)=(2\pi)^{n}\left\langle\tilde a, \varphi_{k-p}^{n+p+q-1}\right\rangle P(z)\varphi_{k-p}^{n+p+q-1}(z).\]
Since $Q_k\not\equiv0$ for infinitely many $k\in\mathbb Z_+.$
Therefore,
\[Q_k^{-1}(0)=P^{-1}(0)\cup\left(\varphi_{k-p}^{n+p+q-1}\right)^{-1}(0)\neq\mathbb C^n,\]
for infinitely many $k.$ In view of Corollary \ref{cor1}, the functions
$\varphi_{k-p}^{n+p+q-1}$'s can not have a common zero except for finitely many
$k\in\mathbb Z_+$ with $k\geq p.$ Hence, we conclude that $S(f)=P^{-1}(0).$
\end{proof}

Now, we would like to address the problem in the higher dimensional space $\mathbb C^n$ with $n\geq2.$
Let $f\in L^2(\mathbb C^n).$ Consider the spherical harmonic decomposition of $f$ as
\begin{equation}\label{exp24}
 f(z)=\sum_{p=0}^\infty\sum_{q=0}^\infty\sum_{j=1}^{d(p,q)}\tilde a_j^{pq}(|z|)P_{pq}^j(z).
\end{equation}
In view of the Hecke-Bochner identities (\ref{exp2}), we conclude that
\begin{eqnarray*}
f\times\varphi_{k_0}^{n-1}
&=&\sum_{p=0}^{k_0}\sum_{q=0}^\infty\sum_{j=1}^{d(p,q)}C_{k_0-p,j}^{pq}P_{pq}^j\varphi_{k_0-p}^{n+p+q-1}\\
&=&\sum_{p=0}^{k_0}\sum_{q=0}^\infty P_{pq}^{k_0}\varphi_{k_0-p}^{n+p+q-1}, \text{ where }P_{pq}^{k_0}\in H_{p,q}.
\end{eqnarray*}
Now look at the following concrete expression for the spectral projections
\begin{equation}\label{exp8}
 Q_k(z)=\sum_{p=0}^{k}\sum_{q=0}^\infty P_{pq}^{k}(z)\varphi_{k-p}^{n+p+q-1}(z),~P_{pq}^{k}\in H_{p,q}.
\end{equation}

\begin{remark}\label{rk4}
$(a).$
As very much similar to complex plane $\mathbb C,$ our believe suggest that any Coxeter
system of hyperplanes can be a set of injectivity for the TSM on $\mathbb C^n.$
For instance on $\mathbb C^2,$ suppose the function
\[Q(z_1, z_2)=\sum_{p=1}^2\left(a_pz_1^p+b_pz_2^p\right) + \sum_{q=1}^2\left(c_q\bar{z_1}^q+d_q\bar{z_2}^q\right)\]
vanishes on each of co-ordinate axis, i.e.,
$Q(x,0)=Q(ix,0)=Q(0,x)=Q(0,ix)=0, \forall x\in\mathbb R.$ Then
$Q\equiv0.$ As another example, consider a typical polynomial
$P(z_1,z_2)=c z_1^p\bar z_2^q\in H_{p,q}.$ Suppose $P(z_1,x_2)=0,
\forall z_1\in\mathbb C$ and $\forall x_2\in\mathbb R.$ Then
$P\equiv0.$ In view of these arguments, write
\[S=(\mathbb C\times\mathbb R)\cup(\mathbb C\times i\mathbb R)\cup(\mathbb R\times\mathbb C)
\cup (i\mathbb R\times\mathbb C).\]
It is natural to ask, whether the set $S$ can be a set of injectivity for the TSM on
$L^q(\mathbb C^2),$ for $1\leq q\leq2.$

\bigskip

$(b).$ Let us rewrite $Q_kf=f\times\varphi_k^{n-1}.$ From the
explicit expression of $Q_k,$ given by (\ref{exp8}), it follows that
\begin{equation}\label{exp23}
\left\|Q_kf\right\|_2^2=\sum_{p=0}^k\sum_{q=0}^\infty\left\|Y_{pq}^k(f)\right\|_2^2
\left\|\varphi_{k-p}^{n+p+q-1}\right\|_2^2,
\end{equation}
where $Y_{pq}^k(f)$'s are spherical harmonics depending upon $f.$ In the work \cite{SZ},
Stempak and Zienkiewic have established that for $f\in L^r(\mathbb C^n),$ with
$1\leq r<\frac{2(2n+1)}{2n+3},$ the operators $Q_k$ satisfy the estimate
$\|Q_kf\|_2\leq C_k\|f\|_r.$ On the basis of the equality (\ref{exp23}), it is
natural to ask, whether the map $f\rightarrow f\times\varphi_k^{n-1}$ would satisfy
an end point estimate.

\bigskip

$(c).$  Let $\mu$ be a finite Borel measure which is supported on a
curve $\gamma$ and $S$ be a non-empty set in $\mathbb R^2.$ Then the
pair $(\gamma, S)$ is called a Heisenberg Uniqueness pairs (HUP) for
$\mu$ if its Fourier transform $\hat\mu(x,y)=0,~\forall~(x,y)\in S,$
implies $\mu=0.$ In a recent work \cite{HR}, Hedenmalm et al. prove
the following result. Suppose $\mu$ is supported on the hyperbola
$\gamma=\{(x,y): ~xy=1\}$ and $\hat\mu$ vanishes on the lattice-cross
$S=\alpha\mathbb Z\times\{0\}\cup\{0\}\times\beta\mathbb Z.$ Then $\mu=0$
if and only if $\alpha\beta\leq1,$ where $\alpha, \beta\in\mathbb R_+.$
This is a variance of uncertainty principle for Fourier transform.
In view of this and the fact that $\varphi_k^0\times\mu$ is real
analytic, it is natural to ask the following question. Let $\mu$ be a
 finite measure supported on a real analytic curve $\gamma$ and $S$ be
a non-empty set in $\mathbb C.$ Then find all those non-trivial pair
$(\gamma, S)$ such that $\varphi_k^0\times\mu(z)=0, ~\forall z\in S$
and $\forall k\geq0.$ Implies $\mu=0.$ Here, we skip to write further
details about these ideas and they might be appear in the successive
work.

\end{remark}

\bigskip

\noindent{\bf Concluding remarks:}
We would like to point out the key motivation behind  $X$-axis together with
$Y$-axis is a set of injectivity for the TSM on $\mathbb C.$ Consider the function
$Q(z)=c_0z+c_1\bar z+c_2z^3+c_3\bar z^3.$ Suppose $Q(x)=Q(ix)=0, \forall x\in\mathbb R.$
Then $Q\equiv0.$ This result can also be interpreted on the Heisenberg group for the
spherical means given by (\ref{exp22}).
The set $\tilde\Sigma_{2N}=\cup_{l=0}^{2N-1}\{(\omega^lx,t): x, t\in\mathbb R \}$
is a set of injectivity for the spherical means on $\mathbb H^1.$

\bigskip

\noindent{\bf Acknowledgements:}
The author wishes to thank E. K. Narayanan for several fruitful discussions, specially,
during my short visit to IISc, Bangalore. The author would like to extend a sincere thank
to M. L. Agranovsky for his reasonable suggestion and remarks. The author would also like
to gratefully acknowledge the support provided by the Department of Atomic Energy,
government of India.

\end{document}